\documentclass{amsart}

\usepackage{hyperref,amsrefs}
\usepackage{txfonts}
\usepackage[all]{xy}

\newtheorem{theorem}{Theorem}[section]
\newtheorem{proposition}[theorem]{Proposition}
\newtheorem{corollary}[theorem]{Corollary}

\theoremstyle{remark}
\newtheorem{remark}[theorem]{Remark}

\theoremstyle{definition}
\newtheorem{definition}[theorem]{Definition}
\newtheorem{notation}[theorem]{Notation}

\newcommand{\Z}{\mathbb{Z}}

\newcommand{\AAA}{\mathbf{\mathcal{A}}}
\newcommand{\CC}{\mathbf{\mathcal{C}}}

\newcommand{\HHH}{\mathcal{H}}

\newcommand{\M}{\mathcal{M}}
\newcommand{\TM}{\mathcal{TM}}

\newcommand{\al}{\alpha}

\newcommand{\ep}{\epsilon}
\newcommand{\ga}{\gamma}
\newcommand{\ka}{\kappa}
\newcommand{\io}{\iota}
\newcommand{\Om}{\Omega}
\newcommand{\phy}{\varphi}

\newcommand{\rla}{\rightleftarrows}

\newcommand{\ral}{\xrightarrow} 

\newcommand{\inj}{\hookrightarrow}
\newcommand{\surj}{\twoheadrightarrow}
\newcommand{\bu}{\bullet}

\newcommand{\We}{\bigvee}
\newcommand{\dfn}{\coloneqq}

\newcommand{\ot}{\otimes}

\newcommand{\x}{\times}

\newcommand{\Ab}{\mathbf{Ab}}

\newcommand{\Gp}{\mathbf{Gp}}
\newcommand{\GrSet}{\mathbf{GrSet}}

\newcommand{\LL}{\mathbf{L}}
\newcommand{\Mod}{\mathbf{Mod}}
\newcommand{\PI}{\mathbf{\Pi}}
\newcommand{\PiAlg}{\mathbf{\Pi Alg}}
\newcommand{\Set}{\mathbf{Set}}

\DeclareMathOperator{\Aut}{Aut}

\DeclareMathOperator{\Der}{Der}
\DeclareMathOperator{\Ext}{Ext}
\DeclareMathOperator{\Fun}{Fun}

\DeclareMathOperator*{\holim}{holim}
\DeclareMathOperator{\Hom}{Hom}

\DeclareMathOperator{\HQ}{HQ}
\DeclareMathOperator{\Hy}{H}
\DeclareMathOperator{\Map}{Map}
\DeclareMathOperator{\Stab}{Stab}

\newcommand{\col}{\mathrm{col}}
\newcommand{\id}{\mathrm{id}}
\newcommand{\inc}{\mathrm{inc}}
\newcommand{\opp}{\mathrm{op}}

\begin{document}

\title{Moduli spaces of 2-stage Postnikov systems} 
\date{}

\author{Martin Frankland}             
\email{franklan@illinois.edu}
\address{Department of Mathematics\\
         University of Illinois at Urbana-Champaign\\
         1409 W. Green St\\
         Urbana, IL 61801\\
         USA}

\thanks{Supported in part by an NSERC Postgraduate Scholarship and an FQRNT Doctoral Research Scholarship.}

\subjclass[2010]{Primary: 55P15; Secondary: 55Q35, 55S45.}

\keywords{moduli space, 2-stage, Postnikov system, realization, $\Pi$-algebra, obstruction, Quillen cohomology, Postnikov truncation, connected cover, classification, $k$-invariant.}

\thanks{This paper expands part of the author's doctoral work at MIT under the supervision of Haynes Miller, whom we thank heartily for all his support. We thank David Blanc, Paul Goerss, Bill Dwyer, Mike Hopkins, and Mark Behrens for fruitful conversations, as well as Jim Stasheff for helpful comments.}

\begin{abstract}
Using the obstruction theory of Blanc-Dwyer-Goerss, we compute the moduli space of realizations of 2-stage $\Pi$-algebras concentrated in dimensions $1$ and $n$ or in dimensions $n$ and $n+1$. The main technical tools are Postnikov truncation and connected covers of $\Pi$-algebras, and their effect on Quillen cohomology.
\end{abstract}

\maketitle

\section{Introduction}

Realization problems abound in algebraic topology. For example, the Steenrod problem asks which unstable algebras over the Steenrod algebra can be realized as the cohomology of a space \cite{Andersen08} \cite{Blanc01}. Along with realization, there is a classification problem: Can one classify all realizations, i.e. all (weak) homotopy types with a given cohomology algebra. Another problem is to classify homotopy types with prescribed homotopy groups. It is a classic result that any graded module over a group $\pi_1$ is realizable by a CW-complex \cite{Whitehead78}*{Chap V.2}. However, taking into account the additional structure of primary homotopy operations makes the problem much more subtle. The algebraic structure encoding all these operations is called a $\Pi$-algebra, which we describe briefly in section \ref{PrimerPi}.

\paragraph{Realization problem:} Given a $\Pi$-algebra $A$, is there a space $X$ satisfying $\pi_*X \simeq A$ as $\Pi$-algebras? In other words, can $A$ be topologically realized, and if so, can we classify all realizations up to weak homotopy equivalence?

Using work of Blanc-Dwyer-Goerss \cite{Blanc04}, we compute the \textit{moduli space} $\TM(A)$ of realizations of certain $\Pi$-algebras $A$. This is an improved classification, in the sense that $\pi_0 \TM(A)$ recovers the usual classification: it is the set of all realizations (weak homotopy types). Then $\pi_1 \TM(A)$ based at a realization $X$ corresponds to automorphisms of $X$, $\pi_2 \TM(A)$ corresponds to automorphisms of automorphisms, and so on.

\subsection{\texorpdfstring{$\Pi$}{Pi}-algebras} \label{PrimerPi}

A $\Pi$-algebra is the algebraic structure best describing the homotopy groups of a pointed space $X$. More details can be found in \cite{Blanc04}*{\texorpdfstring{$\S$}{Section} 4} \cite{Stover90}*{\texorpdfstring{$\S$}{Section} 4}; we recall the essentials. Let $\PI$ denote the homotopy category of pointed spaces with the homotopy type of a finite (possibly empty) wedge of spheres of positive dimensions.

\begin{definition} \label{DefPiAlg}
A $\PI$\textbf{-algebra} is a contravariant functor $A \colon \PI \to \Set$ that sends wedges to products, i.e. a product-preserving functor $\PI^{\opp} \to \Set$ (or equivalently to pointed sets).
\end{definition}

The prototypical example is the functor $[-,X]_*$, the homotopy $\Pi$-algebra $\pi_*X$ of a pointed space $X$. A $\Pi$-algebra $A$ can be viewed as a graded group $\{\pi_i = A(S^i)\}$ (abelian for $i \geq 2$) equipped with primary homotopy operations induced by maps between wedges of spheres, such as precomposition operations $\al^* \colon \pi_k \to \pi_n$ for every $\al \in \pi_n(S^k)$. The additional structure is determined by operations of that form, Whitehead products, and the $\pi_1$-action on higher $\pi_i$, and there are classical relations between them.

\begin{notation}
We write $\PiAlg$ for the category of $\Pi$-algebras, that is $\Fun^{\x}(\PI^{\opp},\Set)$, where $\Fun^{\x}$ denotes product-preserving functors.
\end{notation}

\subsection{Organization and results}

Section \ref{BDG} summarizes the obstruction theory of Blanc-Dwyer-Goerss which we use in later sections. Section \ref{NonSimplyConn} studies the moduli space of realizations of a $\Pi$-algebra concentrated in dimension $1$ and $n$. The main result is theorem \ref{Moduli2types} and its corollary. Section \ref{ConnCoverPiAlg} studies the connected cover functor for $\Pi$-algebras and its effect on Quillen cohomology. The main result is the connected cover isomorphism \ref{HQcover}, an algebraic result which becomes useful later. Section \ref{Stable2Types} studies the moduli space of realizations of a $\Pi$-algebra concentrated in dimension $n$ and $n+1$. The main result is theorem \ref{ModuliStable2types} and its corollary.

\subsection{Notations and conventions}

\begin{notation}
Let $\CC$ be a category of universal algebras. We write $s\CC$ for the category of simplicial objects in $\CC$.
\end{notation}

\begin{definition}
For an object $X$ of $\CC$, the category $\Mod(X)$ of \textbf{Beck modules} over $X$ is the category $Ab(\CC/X)$ of abelian group objects in the slice category $\CC/X$.
\end{definition}

\begin{definition}
\textbf{Abelianization} over $X$ is the left adjoint $Ab_X \colon \CC/X \to Ab(\CC/X)$ of the forgetful functor $U_X \colon Ab(\CC/X) \to \CC/X$ (if it exists).
\end{definition}

\begin{definition}
\textbf{Derivations} over $X$ into an $X$-module $M$ are the functor:
\[
\Hom_{\CC/X}(-,U_X M) = \Hom_{Ab(\CC/X)}(Ab_X(-),M) \colon \CC/X \to \Ab
\]
where $\Ab$ denotes the category of abelian groups.
\end{definition}

\begin{definition}
The \textbf{Quillen cohomology} of $X$ with coefficients in a module $M$ is (simplicially) derived functors of derivations, given by $\HQ^*(X;M) \dfn \pi^* \Hom(Ab_X(C_{\bu}),M)$, where $C_{\bu} \to X$ is a cofibrant replacement of $X$ in $s\CC$. The simplicial module $Ab_X(C_{\bu})$ is called \textbf{cotangent complex} of $X$, denoted $\LL_X$.
\end{definition}

For ease of reading, we sometimes abbreviate ``isomorphism'' to ``iso'' and ``weak equivalence'' to ``weak eq''.

\section{Blanc-Dwyer-Goerss obstruction theory} \label{BDG}

This section summarizes the Blanc-Dwyer-Goerss obstruction theory for the realization of $\Pi$-algebras \cite{Blanc04}*{Thm 1.3, 9.6}. Start with a $\Pi$-algebra $A$ and consider the moduli space $\TM(A)$ of its topological realizations. Try to build it using the moduli spaces $\TM_n(A)$ of ``potential $n$-stages'', which are simplicial spaces that look more and more like realizations of $A$. Here $n$ can be a non-negative integer or $\infty$, and the geometric realization of a potential $\infty$-stage is in fact a realization of $A$.
\begin{itemize}
 \item Geometric realization induces a weak equivalence $\TM_{\infty}(A) \ral{\sim} \TM(A)$.
 \item The maps $\TM_{\infty}(A) \to \TM_n(A)$ induce $\TM_{\infty}(A) \ral{\sim} \holim_n \TM_n(A)$ which is a weak equivalence.
 \item Successive stages $\TM_n(A) \to \TM_{n-1}(A)$ are related in a certain fiber square.
 \item $\TM_0(A)$ is weakly equivalent to $B \Aut(A)$.
\end{itemize}

The interpretation in terms of $\pi_0$ is the following.
\begin{itemize}
 \item A potential $0$-stage exists and is unique (up to weak equivalence).
 \item Given a potential $(n-1)$-stage $Y$, there is a class in Quillen cohomology:
\[
o_Y \in \HQ^{n+2}(A;\Om^n A) / \Aut(A,\Om^n A)
\]
which is the obstruction to lifting $Y$ to a potential $n$-stage.
 \item If $Y$ can be lifted, then different lifts (up to weak equivalence) are classified by $\HQ^{n+1}(A;\Om^n A)$, meaning that this group acts transitively on the set of lifts.
 \item Realizations correspond to successive lifts all the way up to infinity.
\end{itemize}

We can be more precise on the indeterminacy of classifying lifts. Writing $\M(Y)$ for the moduli space of $Y$ i.e. the component of $Y$ in $\TM_{n-1}(A)$, there is a fiber sequence (with correct indexing):
\begin{equation} \label{FiberSeq}
\HHH^{n+1} (A; \Om^n A) \to \TM_n(A)_Y \to \M(Y)
\end{equation}
where $\TM_n(A)_Y$ consists of components of $\TM_n(A)$ sitting over $\M(Y)$, and $\HHH^{n+1} (A; \Om^n A)$ is a Quillen cohomology space\footnote{Quillen cohomology is representable, and what we call a Quillen cohomology space is the derived mapping space into some extended Eilenberg-MacLane object, where derived means taking a cofibrant replacement of the source and fibrant replacement of the target.}, whose $\pi_0$ is the corresponding Quillen cohomology group $\HQ^{n+1} (A; \Om^n A)$. In particular, weak equivalence classes of lifts of $Y$ are in bijection with $\HQ^{n+1}(A;\Om^n A) / \pi_1 \M(Y)$. The Quillen cohomology space satisfies more generally:
\begin{equation} \label{PiHQ}
\pi_i \HHH^k (A; M) = \HQ^{k-i}(A;M)
\end{equation}
for every $i \geq 0$, as explained in \cite{Blanc04}*{6.7}.

There is a relative version $\TM'(A)$ of the moduli space of realizations, where we consider pointed realizations $(X,f)$ of $A$, that is a realization $X$ with an identification $f \colon \pi_*X \simeq A$. The moduli space $\TM(A)$ fibers over $B \Aut(A)$ and the forgetful map $\TM'(A) \to \TM(A)$ sits inside a fiber sequence $\TM'(A) \to \TM(A) \to B \Aut(A)$ \cite{Blanc04}*{1.1}. Moreover, $\TM(A)$ is the homotopy quotient $\TM'(A)_{h \Aut(A)}$.

\section{Homotopy groups in dimensions 1 and n} \label{NonSimplyConn}

In this section, we study the moduli space of realizations of $\Pi$-algebras concentrated in dimensions $1$ and $n$. Let us recall some results from \cite{Frankland10} about Postnikov truncation of $\Pi$-algebras.

\begin{definition}
A $\Pi$-algebra $A$ is called $\mathbf{n}$\textbf{-truncated} if it satisfies $A(S^i) = \ast$ for all $i > n$.
\end{definition}

\begin{notation}
Denote by $\PiAlg_1^n$ the full subcategory of $\PiAlg$ consisting of $n$-truncated $\Pi$-algebras.
\end{notation}

Let $P_n \colon \PiAlg \to \PiAlg_1^n$ denote the Postnikov $n$-truncation of $\Pi$-algebras, which is left adjoint to the inclusion functor. The simplicial prolongation $P_n \colon s\PiAlg \to s\PiAlg_1^n$ is a left Quillen functor which preserves all weak equivalences and fibrations.

\begin{theorem} (Truncation isomorphism) \label{HQtrunc}
Let $A$ be a $\Pi$-algebra and $N$ a module over $A$ which is $n$-truncated. The Postnikov truncation functor $P_n$ induces a natural iso:
\[
\HQ^*_{\PiAlg_1^n}(P_n A; N) \ral{\cong} \HQ^*_{\PiAlg}(A; N).
\]
\end{theorem}

Now let $A$ be a $\Pi$-algebra concentrated in dimensions $1$ and $n$, that is the data of a group $A_1$ and an $A_1$-module $A_n$ in dimension $n$, and zero elsewhere. There is a unique $\Pi$-algebra with such data. Note that $A$ is always realizable, for instance by the Borel construction $X = EA_1 \x_{A_1} K(A_n,n) \to BA_1$. Using notation of \cite{Blanc04}*{3.1}, let us write $BG(M,i)$ for the extended Eilenberg-MacLane object $EG \x_G K(M,i)$.

\begin{theorem} \label{Moduli2types}
The moduli space of pointed realizations of $A$ is:
\[
\TM'(A) \simeq \Map_{BA_1} \left( BA_1, BA_1(A_n, n+1) \right)
\]
where the right-hand side is the derived mapping space of pointed maps over $BA_1$ \cite{Blanc04}*{3.6}.
\end{theorem}

\begin{proof}
We study the tower in section \ref{BDG} converging to $\TM(A)$. The module $\Om^k A$ is zero for $k \geq n$, so once we reach a potential $(n-1)$-stage, it lifts uniquely up to infinity. More precisely, the maps $\TM_k \to \TM_{k-1}$ are weak equivalences for $k \geq n$, so that $\TM_{\infty} \ral{\sim} \TM_{n-1}$ is a weak equivalence.

Starting with the unique potential $0$-stage, we encounter obstructions in $\HQ^*(A; \Om^k A)$. For $1 \leq k \leq n-2$, these groups are in fact all trivial. Using the truncation iso \ref{HQtrunc}, we have:
\begin{equation*}
\HQ^i(A; \Om^k A) \cong \HQ^i(P_{n-k} A; \Om^k A) = \HQ^i(A_1; \Om^k A)
\end{equation*}
where the abuse of notation $A_1$ means the $\Pi$-algebra with the group $A_1$ in dimension 1 and trivial elsewhere. Such a $\Pi$-algebra admits a cofibrant replacement $C_{\bu}$ whose constituent $\Pi$-algebras are also concentrated in dimension $1$, by taking a cofibrant replacement (or free resolution) of the group $A_1$. Since $\Om^k A$ is concentrated in dimension $n-k > 1$, applying the derivations functor $\Hom_{\PiAlg/A} \left( -,\Om^k A \right)$ to $C_{\bu}$ yields a cosimplicial abelian group which is identically zero. It follows that $\HHH^i(A; \Om^k A)$ is in fact contractible for $1 \leq k \leq n-2$, so that the maps:
\[
\TM_{n-2}(A) \to \cdots \to \TM_1(A) \to \TM_0(A) \simeq B \Aut(A)
\]
in the tower are all weak equivalences.

The only possible obstruction is in lifting from the potential $(n-2)$-stage to a potential $(n-1)$-stage. Since $A$ is realizable, the obstruction to existence vanishes. Using the identifications $\TM_{\infty} \ral{\sim} \TM$ and $\TM_{n-2} \ral{\sim} \TM_0 \simeq B \Aut(A)$, we can exhibit the moduli space $\TM$ as the total space of a fiber sequence:
\[
\HHH^n(A; \Om^{n-1} A) \to \TM \to B \Aut(A).
\]
However, the homotopy fiber of the map $\TM \to B \Aut(A)$ is $\TM'$, which is therefore equivalent to the Quillen cohomology space $\HHH_{\PiAlg}^n(A; \Om^{n-1} A) \cong \HHH_{\Gp}^n(A_1; A_n)$. Working with connected pointed spaces instead of simplicial groups (via the Quillen equivalence provided by Kan's loop group), the space is equivalent to $\Map_{BA_1} \left( BA_1, BA_1(A_n, n+1) \right)$.
\end{proof}

Let us represent the obstruction theory schematically as a realization tree for $A$. The nodes at each stage correspond to $\pi_0 \TM_n(A)$ and an edge goes up from each potential $n$-stage to its possible lifts to potential $(n+1)$-stages. In the case at hand, the realization tree looks like this:
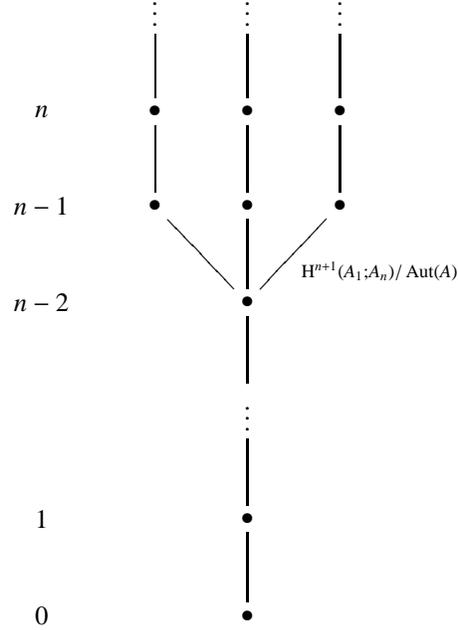
\[
\xymatrix{
& \vdots & \vdots & \vdots \\
n & \bu \ar@{-}[u] & \bu \ar@{-}[u] & \bu \ar@{-}[u] \\
n-1 & \bu \ar@{-}[u] & \bu \ar@{-}[u] & \bu \ar@{-}[u] \\
n-2 & & \bu \ar@{-}[ul] \ar@{-}[u] \ar@{-}[ur]_{\Hy^{n+1}(A_1; A_n)/ \Aut(A)} & \\
  & & \vdots \ar@{-}[u] & \\
1 & & \bu \ar@{-}[u] & \\
0 & & \bu \ar@{-}[u] & \\
}
\]

\begin{corollary} \label{CorModuli2types}
The moduli space of realizations $\TM(A) \simeq \TM'(A)_{h \Aut(A)}$ is non-empty and its path components are $\pi_0 \TM(A) \simeq \Hy^{n+1}(A_1; A_n) / \Aut(A)$. For any choice of basepoint, its homotopy groups are:
\[
\pi_i \TM(A) \simeq \begin{cases}
 0 \; \text{ for } i > n \\
 \Der(A_1,A_n) \; \text{ for } i = n \\
 \Hy^{n+1-i}(A_1; A_n) \; \text{ for } 2 \leq i < n \\
                  \end{cases}
\]
and $\pi_1 \TM(A)$ is an extension by $\Hy^n(A_1; A_n)$ of a subgroup of $\Aut(A)$ corresponding to realizable automorphisms.
\end{corollary}

\begin{proof}
The moduli space $\TM$ is the total space of a fiber sequence:
\[
\HHH_{\Gp}^n(A_1; A_n) \to \TM \to B \Aut(A).
\]
By the long exact sequence of homotopy groups, we obtain $\pi_i \HHH_{\Gp}^n(A_1; A_n) \ral{\simeq} \pi_i \TM$ for all $i \geq 2$. This implies:
\[
\pi_i \TM \simeq \HQ_{\Gp}^{n-i}(A_1; A_n) = \Hy^{n-i+1}(A_1; A_n)
\]
which proves the claim in that range. The bottom part of the long exact sequence is:
\begin{align*}
\pi_2 B \Aut(A) \to & \pi_1 \HHH_{\Gp}^n(A_1; A_n) \to \pi_1 \TM \to \pi_1 B \Aut(A) \to \\
& \to \pi_0 \HHH_{\Gp}^n(A_1; A_n) \to \pi_0 \TM \to \pi_0 B \Aut(A)
\end{align*}
which we can write as:
\[
0 \to \Hy^n(A_1; A_n) \to \pi_1 \TM \to \Aut(A) \to \Hy^{n+1}(A_1; A_n) \to \pi_0 \TM \to \ast
\]
remembering that $\Hy^{n+1}(A_1; A_n)$ is really a pointed set, on which $\Aut(A)$ acts. The sequence gives $\pi_0 \TM \simeq \Hy^{n+1}(A_1; A_n) / \Aut(A)$. The kernel of $\Aut(A) \to \Hy^{n+1}(A_1; A_n)$ is precisely the stabilizer of the basepoint $\ka \in \Hy^{n+1}(A_1; A_n)$, so we get a short exact sequence of groups:
\[
\Hy^n(A_1; A_n) \inj \pi_1 \TM \surj \Stab(\ka)
\]
which proves the claim on $\pi_1 \TM$.
\end{proof}

\begin{remark}
Picking an actual $k$-invariant $\ka \in \Hy^{n+1}(A_1; A_n)$ is the same as choosing an (isomorphism class of) identification $f \colon \pi_*X \simeq A$ for the realization $X$ corresponding to $\ka$. The action of $\phy \in \Aut(A)$ on $\Hy^{n+1}(A_1; A_n)$ corresponds to the postcomposition action $\phy \cdot (X,f) = (X, \phy f)$. The stabilizer of $(X,f)$ is the subgroup of all $\phy \in \Aut(A)$ satisfying $(X,f) = \phy \cdot (X,f) = (X, \phy f)$ in that set of isomorphism classes. In other words, there is a self weak equivalence $h \colon X \simeq X$ realizing the automorphism:
\[
\xymatrix{
\pi_*X \ar[d]_{\pi_* h} \ar[r]^f & A \ar[d]^{\phy} \\
\pi_*X \ar[r]_f & A. \\
}
\]
That is why $\Stab(\ka)$ corresponds to the subgroup of realizable automorphisms of $A$, once all basepoints have been chosen.
\end{remark}

Thus, realizations of $A$ are classified by group cohomology $\Hy^{n+1}(A_1;A_n)$. More precisely, their weak equivalence classes are in bijection with $\Hy^{n+1}(A_1;A_n) / \Aut(A)$, which can be viewed as the $k$-invariant up to its indeterminacy. Theorem \ref{Moduli2types} and its corollary promote this $\pi_0$ statement to a \textit{moduli} statement: The \textit{moduli space} of realizations is weakly equivalent to the \textit{mapping space} where the $k$-invariant lives, up to the action of $\Aut(A)$. The moduli statement encodes information about higher automorphisms of realizations of $A$.

\begin{remark}
The case $n=2$ recovers a classic result of Whitehead and Mac Lane on the classification of connected homotopy $2$-types using crossed modules \cite{MacLane50}*{Thm 1, 2}.
\end{remark}

\section{Connected covers of \texorpdfstring{$\Pi$}{Pi}-algebras} \label{ConnCoverPiAlg}

In \cite{Frankland10}, we studied the Postnikov truncation of $\Pi$-algebras. In this section, we present a very similar story for connected covers of $\Pi$-algebras.

\begin{definition}
A $\Pi$-algebra $A$ is called $\mathbf{n}$\textbf{-connected} if it satisfies $A(S^i) = \ast$ for all $i \leq n$.
\end{definition}

\begin{notation}
Denote by $\PiAlg_{n+1}^{\infty}$ the full subcategory of $\PiAlg$ consisting of $n$-connected $\Pi$-algebras.
\end{notation}

\begin{notation}
Denote by $\PI_{n+1}^{\infty}$ the full subcategory of $\PI$ consisting of spaces with the homotopy type of a wedge of spheres of dimension at least $n+1$, and let $I^n \colon \PI_{n+1}^{\infty} \to \PI$ be the inclusion functor.
\end{notation}

One can go the other way, by removing spheres below a certain dimension. Glossing over technicalities, define a ``collapse'' functor $\ga \colon \PI \to \PI_{n+1}^{\infty}$ by $\ga \left( \We_{i=1}^k S^{n_i} \right) = \We_{n_i > n} S^{n_i}$. It sends a map $f \colon \We_i S^{n_i} \to \We_j S^{m_j}$ to the composite:
\[
\xymatrix{
\We_{n_i > n} S^{n_i} \ar@{^{(}->}[r] & \We_i S^{n_i} \ar[r]^f & \We_j S^{m_j} \ar@{->>}[r] & \We_{m_j > n} S^{m_j}
}
\]
where the last map collapses summands of low dimension.

\begin{proposition}
The inclusion functor $I^n \colon \PI_{n+1}^{\infty} \to \PI$ is right adjoint to the collapse functor $\ga \colon \PI \to \PI_{n+1}^{\infty}$.
\end{proposition}

\begin{proof}
Any object $\We_i S^{n_i}$ in $\PI_{n+1}^{\infty}$ satisfies $\pi_k(\We_i S^{n_i}) = 0$ for $k \leq n$. We conclude:
\[
\Hom_{\PI} \left( \We_{m_j} S^{m_j}, I^n \left( \We_i S^{n_i} \right) \right) \cong \Hom_{\PI_{n+1}^{\infty}} \left( \We_{m_j > n} S^{m_j}, \We_i S^{n_i} \right).
\]
\end{proof}
The unit $1 \to I^n \ga$ is the map collapsing summands of dimension up to $n$, and the counit $\ga I^n \to 1$ is the identity. Note that both $I^n$ and $\ga$ preserve coproducts (wedges).

\begin{proposition} \label{ConnPiAlg}
The functors $\ga^* \colon \Fun^{\x}({\PI_{n+1}^{\infty}}^{\opp},\Set) \cong \PiAlg_{n+1}^{\infty} \colon (I^n)^*$ form an equivalence of categories.
\end{proposition}

In other words, $n$-connected $\Pi$-algebras can be intrinsically defined the same way $\Pi$-algebras are, only using spheres of dimension greater than $n$ instead of all spheres.

\begin{proof}
Passing to opposite categories, $I^n \colon {\PI_{n+1}^{\infty}}^{\opp} \to \PI^{\opp}$ is left adjoint to $\ga$, and both preserve products. Applying the strict $2$-functor $\Fun^{\x}(-,\Set)$ yields $(I^n)^* \colon \Fun^{\x}(\PI^{\opp},\Set) \to \Fun^{\x}({\PI_{n+1}^{\infty}}^{\opp},\Set)$ which is right adjoint to $\ga^*$.

On the one hand, we have $(I^n)^* \ga^* = (\ga I^n)^* = (\id)^* = \id$. On the other hand, $\ga^*$ lands into the full subcategory $\PiAlg_{n+1}^{\infty}$ and the restriction of $\ga^* (I^n)^*$ to $\PiAlg_{n+1}^{\infty}$ is a natural iso. Indeed, if $A$ is an $n$-connected $\Pi$-algebra, we have $\ga^* (I^n)^* A = (I^n \ga)^*A \cong A$. More precisely, $A$ sends all unit maps:
\[
\We_i S^{n_i} \surj \We_{n_i > n} S^{n_i} = I^n \ga (\We_i S^{n_i})
\]
to isos $\prod_{n_i > n} A(S^{n_i}) \ral{\cong} \prod_i A(S^{n_i})$.
\end{proof}

\begin{remark}
Proposition \ref{ConnPiAlg} exhibits the category of $n$-connected $\Pi$-algebras as a category $\mathbf{\Pi_{\AAA} Alg}$ of generalized $\Pi$-algebras in the sense of \cite{Blanc06}*{\S 1.3}, where the set of models is $\AAA = \{ S^i \}_{i=n+1}^{\infty}$. The analogous proposition for $n$-truncated $\Pi$-algebras exhibits them as the category $\mathbf{\Pi_{\AAA} Alg}$ where the set of models is $\AAA = \{ S^i \}_{i=1}^n$.
\end{remark}

According to proposition \ref{ConnPiAlg}, we have the adjunction:
\begin{equation} \label{AdjConnPiAlg}
\xymatrix{
\PiAlg_{n+1}^{\infty} \ar@<0.6ex>[r]^-{\io = \ga^*} & \PiAlg \ar@<0.6ex>[l]^-{C_n = (I^n)^*} \\
}
\end{equation}
where $\io \colon \PiAlg_{n+1}^{\infty} \to \PiAlg$ is the inclusion functor of $n$-connected $\Pi$-algebras into all $\Pi$-algebras and $C_n \colon \PiAlg \to \PiAlg_{n+1}^{\infty}$ is the $n^{\text{th}}$ connected cover functor, equipped with the natural transformation $\ep \colon \io C_n \to 1$.

\begin{remark}
One can check the adjunction (\ref{AdjConnPiAlg}) directly in the ``graded group'' point of view. If $B$ is an $n$-connected $\Pi$-algebra and $A$ is a $\Pi$-algebra, then a map $f$ in $\Hom_{\PiAlg}(\io B, A)$ is the data of group maps $f_i \colon B_i \to A_i$ for all $i > n$ respecting the additional structure. Since the additional structure does not decrease degree, all those conditions live in degree greater than $n$. Therefore the data is exactly that of the corresponding map $f$ in $\Hom_{\PiAlg_{n+1}^{\infty}}(B, C_n A)$.
\end{remark}

Both $\PiAlg$ and $\PiAlg_{n+1}^{\infty}$ are categories of universal algebras -- finitary many-sorted varieties, to be more precise. Since $\io \colon \PiAlg_{n+1}^{\infty} \to \PiAlg$ is full, the free $n$-connected $\Pi$-algebra on a graded set $\{X_i\}_{i > n}$ is the free $\Pi$-algebra on $\{X_i\}$, namely $F \{X_i\} = \pi_* ( \We_i \We_{j \in X_i} S^i )$.

In both categories, projective objects are retracts of free objects and regular epis are surjections of underlying graded sets \cite{Quillen67}*{II.4, Rem 1 after Prop 1}. In particular, the right adjoint $C_n$ preserves regular epis and prolongs to a right Quillen functor. Note also that $\{ \pi_* S^{n+1}, \pi_* S^{n+2}, $\ldots$ \}$ is a set of small projective generators for $\PiAlg_{n+1}^{\infty}$, which exhibits $\PiAlg_{n+1}^{\infty}$ as an algebraic category \cite{Frankland10}*{Def 2.23}. Moreover, the left adjoint $\io \colon \PiAlg_{n+1}^{\infty} \to \PiAlg$ preserves all (small) limits, since they are created on underlying sets in both categories. In particular, $\io$ passes to Beck modules.

\subsection{Standard model structure}

The standard model structure on the category $s\PiAlg$ of simplicial $\Pi$-algebras is described in \cite{Blanc04}*{$\S$ 4.5} and the same description holds for $s\PiAlg_{n+1}^{\infty}$. A map $f \colon X_{\bu} \to Y_{\bu}$ is a fibration (resp. weak eq) if it is so at the level of underlying graded sets -- or equivalently, graded groups. Cofibrations are maps with the left lifting property with respect to acyclic fibrations and can be characterized as retracts of free maps.

\begin{proposition} \label{ioFib}
The left Quillen functor $\io \colon s\PiAlg_{n+1}^{\infty} \to s\PiAlg$ preserves all weak equivalences and fibrations. In particular, it preserves cofibrant replacements.
\end{proposition}

\begin{proof}
\textit{(Functor point of view)} Let $f \colon X_{\bu} \to Y_{\bu}$ be a fibration (resp. weak eq) in $s\PiAlg_{n+1}^{\infty}$. Let $P$ be a projective of $\PiAlg$, exhibited as a retract of a free by $P \to F \to P$. Then $\Hom(P, \io f)$ is a retract of $\Hom(F, \io f)$ so it suffices that the latter be a fibration (resp. weak eq) of simplicial sets.

The $\Pi$-algebra $F = F(S)$ is free on a graded set $S$ and since $X_{\bu}$ is $n$-connected, we have:
\begin{align*}
\Hom_{\PiAlg}(F(S), \io X_{\bu}) &= \Hom_{\GrSet}(S, U \io X_{\bu} ) \\
&= \Hom_{\GrSet}(S_{>n}, U \io X_{\bu} ) \\
&= \Hom_{\PiAlg}(F(S_{>n}), \io X_{\bu} ) \\
&= \Hom_{\PiAlg_{n+1}^{\infty}}(F(S_{>n}), X_{\bu} )
\end{align*}
where $S_{>n}$ denotes the graded set $S$ whose sets up to degree $n$ have been changed to empty sets. Using this, we obtain:
\[
\xymatrix{
\Hom_{\PiAlg}(F, \io X_{\bu}) \ar[d]_{\cong} \ar[r]^{(\io f)_*} & \Hom_{\PiAlg}(F, \io Y_{\bu}) \ar[d]^{\cong} \\
\Hom_{\PiAlg_{n+1}^{\infty}}(F(S_{>n}), X_{\bu}) \ar[r]_{f_*} & \Hom_{\PiAlg_{n+1}^{\infty}}(F(S_{>n}), Y_{\bu}). \\
}
\]
Since $f$ is a fibration (resp. weak eq) in $s\PiAlg_{n+1}^{\infty}$, the bottom row is a fibration (resp. weak eq) of simplicial sets.
\end{proof}

\begin{proof}
\textit{(Graded group point of view)} The map $f \colon X_{\bu} \to Y_{\bu}$ is a fibration (resp. weak eq) of simplicial sets in each degree, hence the map $\io f$ is a fibration (resp. weak eq) of simplicial sets in each degree, since $\io$ is only appending zeros in degrees $1$ through $n$.
\end{proof}

\begin{corollary} \label{ioCotCpx}
For any $n$-connected $\Pi$-algebra $B$, the comparison of cotangent complexes $\io (\LL_B) \ral{\sim} \LL_{\io B}$ induced by the adjunction $\io \dashv C_n$ is a weak equivalence.
\end{corollary}

\begin{proof}
By \cite{Frankland10}*{Prop 3.10} and \ref{ioFib}.
\end{proof}

\begin{theorem} (Connected cover isomorphism) \label{HQcover}
Let $B$ be an $n$-connected $\Pi$-algebra and $M$ a module over $\io B$, i.e. a module over $B$ viewed as a $\Pi$-algebra. Then the natural comparison map in Quillen cohomology:
\begin{equation}
\HQ^*_{\PiAlg}(\io B; M) \ral{\cong} \HQ^*_{\PiAlg_{n+1}^{\infty}}(B; C_n M)
\end{equation}
is an isomorphism.
\end{theorem}

\begin{proof}
By \cite{Frankland10}*{Prop 3.11} and \ref{ioFib}, and the fact that the unit map $\eta_B \colon B \to C_n \io B$ is the identity (hence so is $\eta_B^*$).
\end{proof}

Let us describe modules over an $n$-connected $\Pi$-algebra $B$ more explicitly. Think of a module over a $\Pi$-algebra $A$ as an abelian $\Pi$-algebra on which $A$ acts (cf. \cite{Blanc04}*{$\S$ 4.11}), namely the kernel of the split extension $E \surj A$ as opposed to its ``total space'' $E$.

\begin{proposition}
The category $\Mod(B)$ of modules over an object $B$ of $\PiAlg_{n+1}^{\infty}$ is isomorphic to the full subcategory $\Mod(\io B)^{n\text{-conn}}$ of $\Mod(\io B)$ of modules that happen to be $n$-connected.
\end{proposition}

\begin{proof}
Consider the adjunction on modules:
\[
\xymatrix{
\Mod(B) \ar@<0.6ex>[r]^-{\io} & \Mod(\io B) \ar@<0.6ex>[l]^-{\eta_B^* C_n = C_n}
}
\]
induced by the adjunction $\io \dashv C_n$ (\ref{AdjConnPiAlg}). The composite $C_n \io$ is the identity. Moreover, $\io$ lands in $\Mod(\io B)^{n\text{-conn}}$, since $\io$ simply appends zeros in degrees $1$ through $n$. By restricting $C_n$ to that subcategory, we obtain an adjunction $\Mod(B) \rla \Mod(\io B)^{n\text{-conn}}$ where both composites $C_n \io$ and $\io C_n$ are the identity, i.e. an iso of categories.
\end{proof}

\subsection{Compatibility with truncation}

Let us check that Postnikov truncation of $\Pi$-algebras commutes with connected covers, something which seems obvious in the graded group point of view. Consider the commutative diagrams in $\PiAlg$:
\begin{equation} \label{CoverOfTrunc}
\xymatrix{
C_m A \ar[d]_{\ep_A} \ar[r]^{C_m \eta_A} & C_m P_n A \ar[d]^{\ep_{P_n A}} \\
A \ar[r]_{\eta_A} & P_n A \\
}
\end{equation}
coming from the natural transformation $\ep \colon \io C_m \to 1$ and 
\begin{equation} \label{TruncOfCover}
\xymatrix{
C_m A \ar[d]_{\ep_A} \ar[r]^{\eta_{C_m A}} & P_n C_m A \ar[d]^{P_n \ep_A} \\
A \ar[r]_{\eta_A} & P_n A \\
}
\end{equation}
coming from the natural transformation $\eta \colon 1 \to \io P_n$. (We suppress the iota's for ease of reading.)

\begin{proposition}
Diagrams (\ref{CoverOfTrunc}) and (\ref{TruncOfCover}) are the same. More precisely:
\begin{enumerate}
 \item $C_m P_n = P_n C_m$ (strictly)
 \item $C_m \eta_A = \eta_{C_m A}$
 \item $\ep_{P_n A} = P_n \ep_A$.
\end{enumerate}
\end{proposition}

\begin{proof}
(1) Denote by $I_n \colon \Pi_1^n \to \Pi$ the inclusion functor of low-dimensional spheres into all spheres. The diagram of inclusions:
\[
\xymatrix{
\Pi_{m+1}^n \ar[d]_{I_n} \ar[r]^{I^m} & \Pi_1^n \ar[d]^{I_n} \\
\Pi_{m+1}^{\infty} \ar[r]_{I^m} & \Pi \\
}
\]
strictly commutes. Taking opposite categories and applying the strict $2$-functor $\Fun^{\x}(-,\Set)$ yields the strictly commutative diagram:
\[
\xymatrix{
\PiAlg_{m+1}^n & \PiAlg_1^n \ar[l]_-{C_m} \\
\PiAlg_{m+1}^{\infty} \ar[u]^{P_n} & \PiAlg \ar[l]^-{C_m} \ar[u]_{P_n} \\
}
\]
as claimed.

(2) Denote by $T \colon \Pi \to \Pi_1^n$ the ``truncation'' functor that keeps spheres of dimension $1$ through $n$, right adjoint to $I_n$ \cite{Frankland10}*{\S 4.3}. Denote by $\inc \colon I_n T \to 1$ the counit, which is the inclusion of low-dimensional summands. Then the map $\eta_A$ is $A_* (\inc^{\opp})$, where $A_*$ denotes postcomposition by the functor $A$. We have:
\begin{align*}
C_m (\eta_A) &= (I^m \ga)^* (\eta_A) \\
&= (I^m \ga)^* A_* (\inc^{\opp}) \\
&= A_* (I^m \ga)^* (\inc^{\opp}) \\
\eta_{C_m A} &= (C_m A)_* (\inc^{\opp}) \\
&= ((I^m \ga)^* A)_* (\inc^{\opp}) \\
&= A_* (I^m \ga)_* (\inc^{\opp}).
\end{align*}
Suffices to show the equality of natural transformations $(I^m \ga)^* (\inc) = (I^m \ga)_* (\inc)$. For any object $S = \We_i S^{n_1}$ of $\Pi$, the natural transformation $(I^m \ga)^* (\inc)$ at $S$ is:
\[
\We_{m < n_i \leq n} S^{n_i} = (I^m \ga) (I_n T) (S) \ral{(I^m \ga)(\inc_S)} (I^m \ga)(S) = \We_{m < n_i} S^{n_i}
\]
whereas $(I^m \ga)_* (\inc)$ at $S$ is:
\[
\We_{m < n_i \leq n} S^{n_i} = (I_n T) (I^m \ga) (S) \ral{\inc_{(I^m \ga)(S)}} (I^m \ga)(S) = \We_{m < n_i} S^{n_i}
\]
and both maps are the inclusion of low-dimensional summands.

(3) Similar proof. Denote by $\col \colon 1 \to I^m \ga$ the unit, which is the collapse of low-dimensional summands. The map $\ep_A$ is $A_* (\col^{\opp})$. We have:
\begin{align*}
\ep_{P_n A} &= (P_n A)_* (\col^{\opp}) \\
&= ((I_n T)^* A)_* (\col^{\opp}) \\
&= A_* (I_n T)_* (\col^{\opp}) \\
P_n (\ep_A) &= (I_n T)^* (\ep_A) \\
&= (I_n T)^* A_* (\col^{\opp}) \\
&= A_* (I_n T)^* (\col^{\opp}).
\end{align*}
Suffices to show the equality $(I_n T)_* (\col) = (I_n T)^* (\col)$. For any object $S = \We_i S^{n_1}$ of $\Pi$, the natural transformation $(I_n T)_* (\col)$ at $S$ is:
\[
\We_{n_i \leq n} S^{n_i} = (I_n T) (S) \ral{(I_n T)(\col_S)} (I_n T) (I^m \ga)(S) = \We_{m < n_i \leq n} S^{n_i}
\]
whereas $(I_n T)^* (\col)$ at $S$ is:
\[
\We_{n_i \leq n} S^{n_i} = (I_n T) (S) \ral{\col_{(I_n T)(S)}} (I^m \ga) (I_n T)(S) = \We_{m < n_i \leq n} S^{n_i}
\]
and both maps are the collapse of low-dimensional summands.
\end{proof}

The effect of Postnikov truncation and connected covers on Quillen cohomology commutes as well.

\begin{proposition} \label{CompatIso}
Let $A$ be a $\Pi$-algebra that is $m$-connected and $M$ a module over $A$ that is $n$-truncated. Then the following diagram commutes:
\[
\xymatrix{
\HQ^*_{\PiAlg_1^n}(\io P_n A; M) \ar[d]_{\cong} \ar[r]^{\cong} & \HQ^*_{\PiAlg}(\io A; M) \ar[d]^{\cong} \\
\HQ^*_{\PiAlg_{m+1}^n}(P_n A; C_m M) \ar[r]_{\cong} & \HQ^*_{\PiAlg_{m+1}^{\infty}}(A; C_m M) \\
}
\]
where the vertical maps are connected cover isos \ref{HQcover} and the horizontal maps are truncation isos \ref{HQtrunc}.
\end{proposition}

\begin{proof}
Consider the object $P_n \io A = \io P_n A$ of $\PiAlg_1^n$. In the category $s \Mod(P_n \io A)$ of simplicial modules over $P_n \io A$, consider the diagram:
\[
\xymatrix{
\LL_{P_n \io A} & P_n (\LL_{\io A}) \ar[l]_{\sim} & \\
\io (\LL_{P_n A}) \ar[u]^{\sim} & \io P_n (\LL_A) \ar[l]^{\sim} \ar@{=}[r] & P_n \io (\LL_A) \ar[ul]_{\sim} \\
}
\]
where all iota's denote forgetting connectedness. The top map is the comparison of cotangent complexes induced by $P_n \colon \PiAlg \to \PiAlg_1^n$. The bottom map is $\io$ applied to the comparison induced by $P_n \colon \PiAlg_{m+1}^{\infty} \to \PiAlg_{m+1}^n$. The left map is the comparison induced by $\io \colon \PiAlg_{m+1}^n \to \PiAlg_1^n$. The right map is $P_n$ applied to the comparison induced by $\io \colon \PiAlg_{m+1}^{\infty} \to \PiAlg$. The diagram commutes up to homotopy (note that the maps are defined up to homotopy). Applying $\Hom(-,M)$ and taking $\pi_*$ yields the desired commutative diagram of Quillen cohomology.
\end{proof}

\begin{corollary}
Let $A$ be an $m$-connected, $n$-truncated $\Pi$-algebra and $k$ a positive integer. Then proposition \ref{CompatIso} yields the iso:
\[
\HQ_{\PiAlg}^*(\io A; \Om^k (\io A)) \cong \HQ_{\PiAlg_{m+1}^{n-k}}^*(P_{n-k} A; C_m \Om^k A).
\]
\end{corollary}

The corollary is useful because the left-hand side is something we will be interested in, and the right-hand side is much easier to compute.

\section{Homotopy groups in dimensions n and n+1} \label{Stable2Types}

In this section, we study the moduli space of realizations of $\Pi$-algebras concentrated in dimensions $n$ and $n+1$. The case of $2$-types ($n=1$) has been treated in section \ref{NonSimplyConn}, so we assume $n \geq 2$. Stable $2$-types correspond to $n \geq 3$.

Let $A$ be a $\Pi$-algebra concentrated in dimensions $n$ and $n+1$, that is the data of abelian groups $A_n$ and $A_{n+1}$ and a quadratic (stable quadratic for $n \geq 3$) map $q \colon A_n \to A_{n+1}$, corresponding to precomposition $\eta^*$ by the Hopf map $\eta \colon S^{n+1} \to S^n$. By \textit{stable} quadratic, we mean a map $q \colon A_n \ot_{\Z} \Z/2 \to A_{n+1}$. Such a $\Pi$-algebra $A$ is always realizable, which can be shown using algebraic models for the corresponding homotopy types. For example, one may use reduced quadratic modules for the case $n=2$ and stable quadratic modules for $n \geq 3$ \cite{Baues08}*{Prop 8.3}.

\begin{theorem} \label{ModuliStable2types}
The moduli space of pointed realizations of $A$ is connected and its homotopy groups are:
\[
\pi_i \TM'(A) \simeq \begin{cases}
 0 \; \text{ for } i \geq 3 \\
 \Hom_{\Z}(A_n,A_{n+1}) \; \text{ for } i = 2 \\
 \Ext_{\Z}(A_n,A_{n+1}) \; \text{ for } i = 1. \\
                  \end{cases}
\]
\end{theorem}

\begin{proof}
Similar to \ref{Moduli2types}. By the truncation iso \ref{HQtrunc}, all groups $\HQ^*(A; \Om^k A)$ are zero for $k \geq 2$. That means all obstructions to lifting a potential $1$-stage vanish, and the maps $\TM_k \to \TM_{k-1}$ are weak equivalences and so is $\TM_{\infty} \ral{\sim} \TM_1$. Since $A$ is realizable, the unique potential $0$-stage can be lifted to a potential $1$-stage. Using the identifications $\TM_{\infty} \ral{\sim} \TM$ and $\TM_0 \simeq B \Aut(A)$, we can exhibit the moduli space $\TM$ as the total space of a fiber sequence:
\begin{equation}
\HHH^2(A; \Om A) \to \TM \to B \Aut(A).
\end{equation}
Again, the homotopy fiber of the map $\TM \to B \Aut(A)$ is $\TM'$. By the truncation iso and connected cover iso \ref{HQcover}, the fiber is equivalent to the Quillen cohomology space $\HHH_{\PiAlg_n^n}^2(P_n A; C_{n-1} \Om A) \cong \HHH_{\Ab}^2(A_n; A_{n+1})$. We conclude:
\[
\pi_i \TM' \simeq \pi_i \HHH_{\Ab}^2(A_n; A_{n+1}) = \HQ_{\Ab}^{2-i}(A_n; A_{n+1}) = \Ext_{\Z}^{2-i}(A_n; A_{n+1})
\]
for all $i \geq 0$, as claimed.
\end{proof}

\begin{corollary} \label{CorModuliStable2types}
The moduli space of realizations $\TM(A) \simeq \TM'(A)_{h \Aut(A)}$ is connected and its homotopy groups are:
\[
\pi_i \TM(A) \simeq \begin{cases}
 0 \; \text{ for } i \geq 3 \\
 \Hom_{\Z}(A_n,A_{n+1}) \; \text{ for } i = 2 \\
                  \end{cases}
\]
and $\pi_1 \TM(A)$ is an extension of $\Aut(A)$ by $\Ext_{\Z}(A_n,A_{n+1})$. In particular, all automorphisms of $A$ are realizable.
\end{corollary}

\begin{proof}
As in \ref{CorModuli2types}, we have an iso $\pi_i \TM' \ral{\sim} \pi_i \TM$ for $i \geq 2$, and the bottom part of the long exact sequence is:
\begin{align*}
\pi_2 B \Aut(A) \to & \pi_1 \HHH_{\Ab}^2(A_n; A_{n+1}) \to \pi_1 \TM \to \pi_1 B \Aut(A) \to \\
& \to \pi_0 \HHH_{\Ab}^2(A_n; A_{n+1}) \to \pi_0 \TM \to \pi_0 B \Aut(A)
\end{align*}
which we can write as:
\[
0 \to \Ext_{\Z}(A_n; A_{n+1}) \to \pi_1 \TM \to \Aut(A) \to 0 \to \pi_0 \TM \to \ast
\]
to prove the claim on $\pi_1$. All automorphisms of $A$ are realizable because the kernel of the map $\Aut(A) \to \pi_0 \TM' = 0$ is all of $\Aut(A)$.
\end{proof}

Note that $A$ is realizable by a unique (weak) homotopy type. That is because the $\Pi$-algebra data determines the $k$-invariant in this case. Indeed, the quadratic map $\eta^* \colon A_n \to A_{n+1}$ can be identified with the $k$-invariant \cite{Baues08}*{\S 8}.

\begin{remark}
The fact that $\TM(A)$ is highly truncated means that the homotopy types corresponding to those $\Pi$-algebras have few higher automorphisms and can be efficiently modeled by algebraic models.
\end{remark}

In conclusion, let us compare our approach to some of the approaches in \cite{Moller10}*{\S 2}. For one thing, our realization problem takes into account the $\Pi$-algebra data (of primary homotopy operations), not just the underlying graded group. Also, our method classifies homotopy types of realizations as opposed to fiber homotopy types of certain fibrations. Moreover, our approach works well for the non simply connected case. Of course, the approaches presented in \cite{Moller10} are very useful and well adapted to many situations. We hope that the current work provides a complementary point of view, which answers slightly different questions and may be suitable in other situations.


\begin{bibdiv}
\begin{biblist}

\bib{Andersen08}{article}{
author = {Andersen, Kasper},
author = {Grodal, Jesper},
title = {The Steenrod problem of realizing polynomial cohomology rings},
journal = {Journal of Topology},
volume = {1},
number = {4},
pages = {747 \ndash 760},
date = {2008},
}

\bib{Baues08}{article}{
author = {Baues, Hans-Joachim},
author = {Muro, Fernando},
title = {Secondary homotopy groups},
journal = {Forum Mathematicum},
volume = {20},
number = {4},
pages = {631 \ndash 677},
date = {2008},
}

\bib{Blanc01}{article}{
author = {Blanc, David},
title = {Realizing coalgebras over the Steenrod algebra},
journal = {Topology},
volume = {40},
number = {5},
pages = {993 \ndash 1016},
date = {2001},
}

\bib{Blanc04}{article}{
author = {Blanc, David},
author = {Dwyer, William G.},
author = {Goerss, Paul G.},
title = {The realization space of a $\Pi$-algebra: a moduli problem in algebraic topology},
journal = {Topology},
volume = {43},
number = {4},
pages = {857 \ndash 892},
date = {2004},
}

\bib{Blanc06}{article}{
author = {Blanc, David},
author = {Johnson, Mark W.},
author = {Turner, James M.},
title = {On realizing diagrams of $\Pi$-algebras},
journal = {Algebr. Geom. Topol.},
volume = {6},
pages = {763 \ndash 807},
date = {2006},
}

\bib{Frankland10}{article}{
author = {Frankland, Martin},
title = {Behavior of Quillen (co)homology with respect to adjunctions},
eprint = {arXiv:1009.5156v2},
date = {2010},
}

\bib{MacLane50}{article}{
author = {Mac Lane, Saunders},
author = {Whitehead, J.H.C.},
title = {On the 3-type of a complex},
journal = {Proc. Nat. Acad. Sci. U.S.A.},
volume = {36},
number = {1},
pages = {41 \ndash 48},
date = {1950},
}

\bib{Moller10}{article}{
author = {M\o{}ller, Jesper M.},
author = {Scherer, J\'er\^ome},
title = {Can one classify finite Postnikov pieces?},
journal = {Bull. Lond. Math. Soc.},
volume = {42},
number = {4},
pages = {661 \ndash 672},
date = {2010},
}

\bib{Quillen67}{book}{
author = {Quillen, Daniel G.},
title = {Homotopical Algebra},
series = {Lecture Notes in Mathematics},
number = {43},
publisher = {Springer},
date = {1967},
}

\bib{Stover90}{article}{
author = {Stover, Christopher R.},
title = {A van Kampen spectral sequence for higher homotopy groups},
journal = {Topology},
volume = {29},
number = {1},
pages = {9 \ndash 26},
date = {1990},
}

\bib{Whitehead78}{book}{
author = {Whitehead, George W.},
title = {Elements of homotopy theory},
series = {Graduate Texts in Mathematics},
number = {61},
publisher = {Springer-Verlag},
date = {1967},
}

\end{biblist}
\end{bibdiv}
\end{document}